\documentclass[10pt]
{amsart}

\usepackage{amsmath, amssymb}

\newtheorem{theorem}{Theorem}[section]
\newtheorem{lemma}[theorem]{Lemma}
\numberwithin{equation}{section}

\newcommand{\commentOut}[1]{}

\begin{document}

\title{The Distribution of Ramsey Numbers}         

\author{Lane Clark}
\address{Lane Clark: Department of Mathematics, 
Southern Illinois University, Carbondale, IL 62901-4408}
\email{lclark@math.siu.edu } 

\author{Frank Gaitan}
\address{Frank Gaitan: Laboratory for Physical Sciences,
8050 Greenmead Drive, College Park, MD 20740}
\email{fgaitan@lps.umd.edu} 


\date{}          


\begin{abstract}
We prove that the number of integers in the interval $[0,x]$ that are non-trivial Ramsey numbers $r(k,n)$ 
\,$(3 \le k \le n)$ has order of magnitude $\sqrt{x \ln x}$\,.
\end{abstract}

\maketitle

\section{Introduction}      

Suppose function ${\frak f} : A \to \mathbb{N}$ where $A \subseteq \mathbb{N}^k$\,.
Then $\frak f$ has {\it distribution function}  
$D_{\frak f} : [0,\infty) \to \mathbb{N}$ defined by $D_{\frak f} (x) = \#{\frak f}(A) \cap [0,x]$\,. Here
$D_{\frak f}(x)$ counts each value of ${\frak f}$ at most $x$ precisely once ignoring its multiplicity. Related is the
{\it distribution function} $M_{\frak f} : [0,\infty) \to \mathbb{N}$ defined by 
$M_{\frak f} (x) = \#\{a \in A : f(a) \le x\} = \sum_{n \le x} \# \frak f^{-1} (n)$\,.
Here $M_{\frak f}(x)$ counts each value of ${\frak f}$ at most $x$ according to its multiplicity.
Then $D_{\frak f} (x) \le M_{\frak f} (x)$ with equality when $\frak f$ is injective. 
Estimating $D_{\frak f}$ (or $M_{\frak f}$) is fundamental for many counting functions $\frak f$\,.
Here $\mathbb{N}$ (respectively, $\mathbb{P}$) denotes the non-negative (respectively, positive) integers. 
A set $S$ has cardinality $\#S$\,.

The distribution of values of number-theoretic functions is a central research area in number theory. 
The seminal example is the distribution of the prime numbers:
Let ${\frak p} : \mathbb{P} \to \mathbb{N}$ where ${\frak p}(n)$ is the $n^{th}$ prime. The function 
$\pi(x) = D_{\frak p}  (x) = M_{\frak p} (x)$ is the number of primes at most $x$\,.
Chebyshev \cite{Ch} determined the order of magnitude of $\pi(x) = \Theta(x/\ln x)$\,. 
(See {\it Mathematical Reviews} Mathematics Subject Classification 11N for many further examples.)

Combinatoric--theoretic numbers may be viewed as functions ${\frak f} : A \to \mathbb{N}$, where 
$A \subseteq \mathbb{N}^k$, and, hence, their distribution function $D_{\frak f}$ (or $M_{\frak f}$) investigated.
Very few have however: 
The function ${\frak b} : \mathbb{N}^2 \to \mathbb{N}$ where ${\frak b}(n,k) = \binom{n}{k}$ 
has distribution function $D_{\frak b}(x) = \sqrt{2x} + o \big(\sqrt{x}\,\big)$ (cf. [11; pp. 76--77]).
The distribution function $D_{\frak N}$ of the function
${\frak N} : \mathbb{P}^2 \to \mathbb{N}$ where the ${\frak N}(n,k)$ are the Narayana numbers was determined in  \cite{CD}.
%
Erd\"os and Niven \cite{EN} proved that the number of distinct multinomial coefficients at most $x > 0$ is 
$(1+\sqrt{2}\,)\sqrt{x} + o(\sqrt{x}\,)$  (see also \cite{A}).
The multiplicity problem was examined and best possible bounds for $\# \frak f^{-1} (n)$ proved 
for a large class of functions $\frak f$ including $\frak b, \frak N$ in \cite{Cl}. 
%

The Ramsey numbers, $r(k,l)$ where $k,l \in \mathbb{P}$ (cf. 
\cite{W}), are among the most important of combinatoric--theoretic numbers.   
Despite decades of effort our knowledge of the Ramsey numbers is quite limited. 
Only nine non--trivial values of $r(k,l)$, with $3 \le k,l \le 5$, and non-trivial bounds for certain other $r(k,l)$, with $3 \le k,l \le 19$, are known (cf. 
\cite{R}).
The best lower bounds for general $r(k,l)$ are \cite{B}, \cite{BK} while the best upper bounds for general $r(k,l)$ are \cite{CO} at present. 
%
They have very different orders of magnitude. The order of magnitude of only one infinite family of Ramsey numbers is known:
Kim \cite{K} proved that $r(3,n) = \Theta (n^2/\log n)$ by improving the lower bound $r(3,n) \ge c_1 n^2/\log^2 n$ of 
\cite{E} to match the upper bound $r(3,n) \le c_2 n^2/\log n$ of 
\cite{AKS}. 
See also \cite{GRS}.

The computational complexity of determining $r(k,l)$ is not known although clearly hard (cf. 
\cite{H}, 
\cite{S}). 
More is known about certain variations. 
Burr \cite{B1} proved that determining whether the graph Ramsey number $r(G,H) \le m$ is NP-hard.
He \cite{B2} proved that determining whether the arrow relation $F \rightarrow (G,H)$ holds is coNP-hard. 
Schaefer \cite{Sc} proved that determining whether $F \rightarrow (G,H)$ holds is $\Pi_2$--complete. 
A quantum algorithm in complexity class QMA for computing $r(k,l)$ was given by the authors \cite{GC} who proved that its solution can be found using adiabatic evolution
(see also 
\cite{BC}).

Trivially $r(1,n) = r(n,1) = 1$ $(n \ge 1)$ and 
$r(2,n) = r(n,2) = n$  $(n \ge 2)$. Hence every positive integer is a {\it trivial} Ramsey number.
We consider the {\it non-trivial} Ramsey numbers $r(k,n)$ where $3 \le k \le n$ since all $r(k,n) = r(n,k)$\,.
We note that $k_1 \le k_2$ and $n_1 \le n_2$ imply $r(k_1,n_1) \le r(k_2,n_2)$\,. 
The function ${\frak r} : A \to \mathbb{N}$, where $A = \{(k,n) \in \mathbb{P}^2 : 3 \le k \le n\}$, defined by 
${\frak r}(k,n) = r(k,n)$ gives the non-trivial Ramsey numbers. In this note we prove that its distribution function 
$D_{\frak r}(x) = \Theta (\sqrt{x \ln x}\,)$\,. 
It is surprising that this 
fundamental property of Ramsey numbers 
can be determined at present. 
Our result for Ramsey numbers is the 
analog of Chebyshev's \cite{Ch} result for prime numbers. 

\section{Distribution of the Ramsey Numbers}

There exists a constant $c > 0$ and an integer $N_1 \ge 4$, such that
\begin{flalign*}
&& && r(4,n) \ge c n^{5/2}/\ln^2 n && (n \ge N_1) &&& \tag{2.1}
\end{flalign*}   
(cf. \cite{B}). It follows that there exists a constant $1 \ge d > 0$ such that
\begin{flalign*}
&& && r(4,n) \ge d n^{9/4}\,. && (n \ge 4) &&& \tag{2.2}
\end{flalign*}  
Further $r(k,n) \ge r(4,n)$ for all $n \ge k \ge 4$\,. Hence
\begin{flalign*}
&& && r(k,n) \ge d n^{9/4}\,.  &&(n \ge k \ge 4) &&& \tag{2.3}
\end{flalign*}

Define
\begin{align*}
&R_k := \{r(k,n) : k \le n\} \quad\quad (k \ge 3) \tag{2.4a} \\  
&R := \{r(k,n) : 3 \le k \le n\} = \bigcup_{k=3}^{\infty} R_k\,. \tag{2.4b}
\end{align*} 
%
Hence
\begin{equation}
R_3 \subseteq R = R_3 \cup \bigcup_{k=4}^{\infty} R_k\,. \tag{2.5}
\end{equation}

For $x \in [0,\infty)$, define 
\begin{align*}
&R_k(x) := R_k \cap [0,x], &&r_k(x) := \# R_k(x) \tag{2.6a} \\ 
&R(x) := R \cap [0,x], &&r(x) := \# R(x)\,. \tag{2.6b} 
\end{align*}    
Then $R(x)$ is the set of non-trivial Ramsey numbers
and $r(x)$ is the number of non-trivial Ramsey numbers, ignoring their multiplicity, 
in the interval $[0,x]$\,. Notice that $r(x) = D_{\frak r}(x)$ from the introduction.
Then (2.5) gives
\begin{equation}
\quad R_3(x) \subseteq R(x) = R_3(x) \cup \bigcup_{k=4}^{\infty} R_k(x)\,, \tag{2.7}
\end{equation} 
hence,
\begin{equation}
r_3(x) \le r(x) \le r_3(x) + \sum_{k=4}^{\infty} r_k(x)\,. \tag{2.8}
\end{equation}

\begin{lemma}
Suppose $k \ge 4$. If $x \ge d^{81/16} k^{9/4}$, then 
$r_k(x) \le d^{-9/4} x^{4/9}$\,. 
\end{lemma}

\begin{proof}
If $x \ge d^{81/16} k^{9/4}$, then $d^{-9/4} x^{4/9} \ge k$\,.
Suppose integer $l > d^{-9/4} x^{4/9}$ $(\ge k)$\,.
Inequality (2.3) gives $r(k,l) \ge d l^{9/4} > x$\,.
Hence $r(k,l) \notin R_k(x)$\,. This implies the result.
\end{proof} 

\begin{lemma}
Suppose $x \ge 8$\,. If $k \ge \lceil2 \log_2 x\rceil \, (\ge 3)$, then
$r_k(x) = 0$\,.
\end{lemma}

\begin{proof} 
Suppose $k \ge \lceil2 \log_2 x\rceil$\,. Erd\"os' classic result gives 
$r(k,n) \ge r(k,k) > 2^{k/2} \ge 2^{\log_2 x} = x$\,. This implies the result.  
\end{proof}

The result of 
\cite{K} implies there exist positive constants $c_1 < c_2$ such that
\begin{flalign*}
&& && c_1 n^2/\ln n \le r(3,n) \le c_2 n^2/\ln n\,. && (n \ge 3) &&& \tag{2.9} 
\end{flalign*}
Inequality (2.9) implies there exists $x_1 \ge 6$ such that
\begin{flalign*}
&& && c_3 \sqrt{x \ln x} \le r_3 (x) \le  c_4 \sqrt{x \ln x} && (x \ge x_1) &&& \tag{2.10}
\end{flalign*}
where, say $c_3 = (4c_2)^{-1/2} > 0$ and $c_4 = c_1^{-1/2} > 0$, which is adequate for our needs.\\


\begin{theorem}
We have $r(x) = D_{\frak r}(x) = \Theta  \big( \sqrt{x \ln x}\,\big)$\,.
\end{theorem}

\begin{proof}
There exists $x_2 \ge 8$ such that 
\begin{flalign*}
&& && x \ge 2^{9/4} d^{81/16}\log_2^{9/4}x\,. && (\forall x \ge x_2) &&& \tag{2.11}
\end{flalign*}
Fix $x \ge \max\{x_1,x_2\}$ with $x_1$ from (2.10) and $x_2$ from (2.11). Then (2.8) and (2.10) give
\begin{equation}
c_3 \sqrt{x \ln x} \le r_3(x) \le r(x) \le r_3(x) + \sum_{k=4}^{\infty} r_k(x)
\le c_4 \sqrt{x \ln x} + \sum_{k=4}^{\infty} r_k(x)\,.  \tag{2.12}
\end{equation} 
Each $r_k(x) = 0$ for $k \ge \lceil2 \log_2 x\rceil$ by Lemma 2.2. Hence
\begin{equation}
\sum_{k=4}^{\infty} r_k(x) = \sum_{k=4}^{\lfloor2 \log_2 x\rfloor} r_k(x)\,.  \tag{2.13}
\end{equation}
If $x \ge 2^{9/4} d^{81/16}\log_2^{9/4} x$, then 
\begin{flalign*}
&& && x \ge d^{81/16}k^{9/4}\,. && (k = 4,\dots,\lfloor2 \log_2 x\rfloor) &&&\tag{2.14}
\end{flalign*}
Lemma 2.1 gives
\begin{equation} 
\sum_{k=4}^{\lfloor2 \log_2 x\rfloor} r_k(x)
\le \sum_{k=4}^{\lfloor2 \log_2 x\rfloor} d^{-9/4} x^{4/9}
\le  2 d^{-9/4} (\log_2 x) x^{4/9}\,. \tag{2.15} 
\end{equation}
For $x \ge \max\{x_1,x_2\}$, (2.12), (2.13) and (2.15) give
\[
c_3 \sqrt{x \ln x} \le r(x) \le c_4 \sqrt{x \ln x} + 2 d^{-9/4} (\log_2 x) x^{4/9}  \tag{2.16} 
\] 
which implies our result.  
\end{proof}

\vspace{-.1in}

\section{Conclusion}

We have proved that the number of non-trivial Ramsey numbers at most $x$ is 
$\Theta (\sqrt{x \ln x}\,)$\,. It is interesting that this fundamental fact can be determined at present. 
Our result for the Ramsey numbers is the analog of Chebyshev's \cite{Ch} result for the prime numbers.
The density of non-trivial Ramsey numbers is then roughly the square root of the density of prime numbers.

As noted in the introduction, very little work has been done to determine the distribution of other significant 
families of combinatoric-theoretic numbers. We think this is an interesting, and important, direction for future work.




\end{document}